\documentclass[12pt,oneside]{amsart} 
\usepackage{amssymb,amscd}
\usepackage{euscript}

      \theoremstyle{plain}
      \newtheorem{theorem}{Theorem}[section]
      \newtheorem{lemma}[theorem]{Lemma}
      
      \newtheorem{proposition}[theorem]{Proposition}

      \newtheorem{definition}[theorem]{Definition}      

\numberwithin{equation}{section}

      \makeatletter
      \def\@setcopyright{}
      \def\serieslogo@{}
      \makeatother

\def\R{\mathbb R}
\def\Z{\mathbb Z}

\def\diam{\text{diam}}
\def\Id{\text{Id}}
\def\d{\delta}
\def\e{\varepsilon}

\def\od{\overline{\dim}}
\def\ud{\underline{\dim}}
\def\tr{\tilde{r}}
\def\Ci{C^\infty}

\def\QED{\hfill\hfill{\square}}

\begin{document}

\date{August 31, 2008}
\author{Victoria Sadovskaya$^\ast$}

\address{Department of Mathematics $\&$ Statistics, 
 University of South  Alabama, Mobile, AL 36688, USA}
\email{sadovska@jaguar1.usouthal.edu}

\title [Dimensional characteristics of invariant  measures for circle diffeos $\;\;$]
{Dimensional characteristics of invariant  measures for circle diffeomorphisms} 

 \thanks{$^\ast$ Supported in part by NSF grant DMS-0401014}


\begin{abstract} We consider pointwise, box, and Hausdorff dimensions
of invariant measures for  circle diffeomorphisms. 
We discuss the cases of rational, Diophantine, and Liouville 
rotation numbers. Our main result is that for any Liouville 
number $\tau$ there exists a $C^\infty$ circle diffeomorphism 
with rotation number $\tau$ such that the pointwise and box
dimensions of its unique invariant measure do not exist.
Moreover, the lower pointwise and lower box dimensions 
can equal any value $0\le \beta \le 1$.
\end{abstract}

\maketitle 


\section{Introduction}

The study of dimensional characteristics of invariant sets and
measures was originated by physicists and applied mathematicians
in the context of strange attractors.
Beginning with the work of Eckmann and Ruelle \cite{ER} 
it developed into a rigorous mathematical theory. Dimension 
theory now plays an important role in dynamics \cite{P}.
Dimensional properties of invariant sets and measures are 
often related to other characteristics of the dynamical system,  
such as Lyapunov exponents and entropy.

In this paper we study pointwise, box, and Hausdorff dimensions
of invariant measures for circle diffeomorphisms. The notion of 
pointwise (or local) dimension was introduced by Young in \cite{Y}.
It plays an important role in  dimension theory of dynamical systems.
For a Borel measure $\mu$ on a metric space $X$,  
its {\em lower and upper pointwise dimensions} at  
a point $x$ are defined as 
$$ 
  \underline{d}_{\mu}(x)=\liminf_{r\to 0}\, 
 \frac{\log{\mu(B(x,r))}}{\log{r}} \quad \text{ and } 
  \quad   \overline{d}_{\mu}(x)=\limsup_{r\to 0}\, 
  \frac{\log{\mu(B(x,r))}}{\log{r}}, 
$$
where $B(x,r)$ is a ball of radius $r$ centered at $x$. 
If the two limits coincide, then their common value $d_\mu(x)$ 
is called the {\em pointwise dimension}\, of $\mu$ at $x$.
The pointwise dimension describes the local distribution of the measure and 
of a typical orbit. It serves as an important 
tool for estimating the Hausdorff and box dimensions of 
measures and sets, see Section \ref{Hausdorff, box}. 
For example, if the pointwise dimension of an ergodic measure $\mu$
exists almost everywhere, then all dimensional characteristics  
of $\mu$ coincide and give a fundamental characteristic of $\mu$ 
called the fractal dimension. 
 
In \cite{BPS}  Barreira,  Pesin, and Schmeling
showed that pointwise dimension exits for any hyperbolic invariant 
measure of a $C^{1+\alpha}$ diffeomorphism.
Dimensional structures of non-hyperbolic measures can be more 
complicated. In \cite{LM} Ledrappier and Misiurewicz constructed
an example of a $C^r$ map of an interval preserving an ergodic
measure whose pointwise dimension does not exist almost 
everywhere. A natural class of non-hyperbolic measures is given 
by invariant measures for circle diffeomorphisms with irrational
rotation numbers. In  \cite{KS} we constructed examples of such 
diffeomorphisms for which pointwise dimension of the measures
does not exist almost everywhere. 

\vskip.1cm
  
Dimensional properties of an invariant measure of a circle 
diffeomorphism $f$ depend significantly on the rotation
number of $f$ (see Section \ref{rotation}). First we  
consider the simpler cases of  rational and  Diophantine
numbers, and then describe our main result for Liouville  numbers.
The rotation number of $f$ is rational if and only 
if $f$ has periodic points. Such a diffeomorphism may
preserve a variety of measures with different properties,
however, any {\em ergodic} invariant measure 
for $f$ is a uniform $\delta$-measure on a periodic orbit.
This immediately implies the following result.

\begin{proposition} \label{rational} 
Let $f$ be a circle homeomorphism with a rational rotation 
number and let $\mu$ be an {\em ergodic} invariant measure 
for $f$. Then 
    \begin{enumerate}
        \item $d_{\mu}(x)=0$  for $\mu$-almost every $x$ in $S^1$,
        \item $\dim_H \mu =\ud_{B\,} \mu = 
                   \od_{B\,} \mu=0.$    
\end{enumerate}
\end{proposition}

In contrast, diffeomorphisms with an irrational rotation number are 
uniquely ergodic.  In this case, the properties of the invariant measure 
depend on how well the irrational rotation number can be approximated 
by rational numbers. The numbers that can not be rapidly approximated 
by rationals are called Diophantine.

\begin{definition} A number $\tau$ is called  {\em Diophantine} 
if  there exist $\delta>0$ and $K>0$ such that 
\begin{equation}\label{dioph}
    | \,\tau - p/q\, | \,>\, K / |q|^{2+\delta}\;\text{ for any integers }
    p\text{ and } q.
\end{equation}
\end{definition}

Circle diffeomorphisms with Diophantine rotation numbers are
 smoothly conjugate to rotations. Therefore
the invariant measure for such
a diffeomorphism is equivalent to the Lebesgue measure
and hence  has the same dimensional properties.

\begin{proposition} \label{Diophantine}
Let $f$ be a $C^\infty$ circle homeomorphism with a Diophantine 
rotation number. Then for its unique invariant measure $\mu$,
    \begin{enumerate}
        \item $d_{\mu}(x)=1$ for every $x$ in $S^1$,
        \item $\dim_H \mu =\ud_{B\,} \mu = 
                   \od_{B\,} \mu=1.$    
\end{enumerate}
\end{proposition}
\newpage

The most interesting case is that of Liouville rotation numbers.
These are irrational numbers that can be rapidly approximated 
by rationals, more precisely:

\begin{definition}
An irrational number  $\tau$ is called a {\em Liouville number} if 
for any $n\ge 1$ there exist integers $p$ and $q$, $\,q>1$, such that 
\begin{equation}\label{liouville}
 |\, \tau - p/q\, |< 1/ q^n.
\end{equation}
\end{definition}

Clearly, an irrational number is Liouville if and only if it is not Diophantine. 
The set  of all Liouville numbers is a dense $G_\delta$ set in $\R$, and 
it has zero Lebesgue measure.
Our main result, Theorem 1.5,  shows that in the case of
a Liouville rotation number different types of dimensional properties
of the invariant measure are realized. In particular, pointwise and 
box dimensions may not exist, which is in contrast to the Diophantine 
case as well as to the case of hyperbolic measures. 

\begin{theorem}\label{main} 
 Let $\tau$ be a Liouville number and let $0\le\beta \le 1$.
There exists a $C^\infty$ circle diffeomorphism $f$ with rotation 
number $\tau$ such that for its unique  invariant measure $\mu$,
    \begin{enumerate}
        \item $\underline{d}_{\mu}(x)=\beta$ and 
        $\,\overline{d}_{\mu}(x)=1$ 
           for $\mu$-almost every $x$ in $S^1,$
        \item $\dim_H \mu =\ud_{B\,} \mu =\beta$ and 
            $\;\od_{B\,} \mu =1.$ 
    \end{enumerate}
 \end{theorem}
 
 It is an interesting open question whether there exists a circle 
 diffeomorphism with irrational rotation number
 whose invariant measure has upper pointwise dimension less than 1
 on a set of positive measure.
 \vskip.1cm

Our constructions are based on a method developed by  Anosov 
and  Katok in \cite{AK} to produce examples of diffeomorphisms 
with specific ergodic properties. In \cite{KS} we used this method
to construct  examples of diffeomorphisms satisfying (1) and (2)
of Theorem \ref{main}.  However,  the qualitative nature of the 
arguments did not allow us to construct  examples for a given 
rotation number, or even describe explicitly the rotation numbers 
in our examples. 
In this paper we use some ideas developed in \cite{FS,FSW} to
make an explicit construction with specific quantitative estimates. 
This allows us to produce the examples for {\em all}$\,$ 
Liouville rotation numbers.

\vskip.1cm

We note that any $C^2$ circle diffeomorphism $f$ with irrational 
rotation number is topologically conjugate to the corresponding rotation. 
The conjugacy gives the distribution function of the invariant measure $\mu$. 
In the theorem above, $\mu$ is singular for $\beta <1$, and so is the conjugacy.
Thus the theorem implies that for any Liouville rotation number there exists a diffeomorphisms with singular conjugacy. Similar methods may be used
to construct diffeomorphisms with specific degree of regularity of the
conjugacy for any Liouville rotation number.

\newpage

\section{preliminaries} \label{preliminaries} 

\subsection{Rotation number of a circle homeomorphism} 
\label{rotation}
(See \cite{KH} for more details.)$\;$ \\
Let $f$ be an orientation-preserving homeomorphism of $S^1$,
let $\pi: \R\to S^1=\R/ \Z$ be the natural projection, and let $F$
be  a homeomorphism of $\R$ such that $f\circ \pi = \pi \circ F$.  
Then the following limit exists and has the same value for all $x$:
$$
  \tau (F)=\lim_{|n| \to \infty} \text{\footnotesize $\frac{1}n$}
   \left( F^n(x)-x \right). 
 $$ 
The number $\tau(f)=\pi(\tau(F))$ is called the {\em rotation number} of $f$.
If $h:S^1 \to S^1$ is a homeomorphism,
then $\tau(h^{-1}\circ f \circ h)=\tau(f)$. In particular, if
 $f$ is topologically conjugate to a rotation by $\tau$
then $f$  has rotation number $\tau$. 


 \subsection{Hausdorff and box dimensions of sets and  measures.} 
 \label{Hausdorff, box} 
 (See \cite{P} for more details.) 
 The {\em upper and lower box dimensions}\, of a set 
$Z \subset \R^k$ are defined as 
 $$
  \od_{B\,} Z= \limsup_{\e\to 0}
   \frac{\log N(Z,\e)}{\log (1/\e)} \quad  \text { and } \quad
   \ud_{B\,} Z= \liminf_{\e\to 0}  \frac{\log N(Z,\e)}{\log (1/\e)},        
 $$
where $N(Z,\e)$ is  the least number 
of balls of diameter $\e$ needed to cover $Z$.

For a number $\alpha\geq 0$, the $\alpha$-Hausdorff measure of $Z$ is 
 $$
  m_H(Z,\alpha)= \lim_{\e\to 0}\,
  \inf_{\EuScript{G}} {\sum}_{U\in\EuScript{G}}(\diam \,U)^\alpha, 
 $$
where the infimum is taken over all finite or countable coverings 
$\EuScript{G}$ of $Z$ by open sets with diameter at most $\e$. 
The {\em Hausdorff dimension}\, of $Z$  is 
 $$
  \dim_H Z=\inf\,\{ \alpha :\;m_H(Z,\alpha)=0 \} =
  \sup\,\{\alpha :\; m_H(Z,\alpha)=\infty \}.   
 $$

The  Hausdorff and upper and lower box dimensions of a  Borel probability
measure $\mu$ are defined as follows:
 \begin{equation} \label{dim_mu}
      \aligned
      \dim_H\mu &=  \inf\, \{\, \dim_H Z: \; \mu(Z)= 1\, \},  \\  
      \ud_{B\,}\mu &= \lim_{\varepsilon \to 0} \, \inf \,    
         \{\, \ud_{B\,} Z: \; \mu(Z)>1-\varepsilon \, \},   \\  
      \od_{B\,}\mu &= \lim_{\varepsilon\to 0} \,\inf \,
         \{\, \od_{B\,} Z: \; \mu(Z)>1-\varepsilon \,\}.
      \endaligned
\end{equation}
It is known that 
$\dim_H \mu \leq \ud_{B\,} \mu \leq \od_{B\,} \mu$.
 \vskip.2cm
  
The following result by L.-S. Young  \cite{Y}
shows how the pointwise dimension of a measure 
can be used to estimate its  box and Hausdorff dimensions.

\begin{theorem}   \label{Young}
Let $\mu$ be a Borel finite measure on $\R^m$. Then
\begin{enumerate}
\item If $\,\underline{d}_\mu(x) \ge d$ for $\mu$-almost every $x$
        then $\,\dim_H \mu \ge d$;
\item If $\,\overline{d}_\mu(x) \le d$ for $\mu$-almost every $x$
        then $\,\od_{B\,} \mu \le d$;
\item If  $\,\underline{d}_\mu(x) = \overline{d}_\mu(x) = d\,$
        for $\mu$-almost every $x,\,$ then \\
        $\dim_H \mu = \ud_{B\,} \mu = 
        \od_{B\,} \mu =d$.
\end{enumerate}

\end{theorem}


\section{Proofs}  \label{proofs}


Throughout this paper we will identify the unit circle $S^1=
\R/\Z$ with the interval $[0,1]$.
Let $\tau$ be an irrational number and let $R_\tau$ be the 
rotation by $\tau$. The Lebesgue measure is the
only measure preserved by $R_\tau.$
Suppose that  
    $$f=h^{-1} \circ R_\tau \circ h,
    $$
 where $h$ is a homeomorphism. Then the unique invariant measure $\mu$ 
for $f$ is the push-forward of the Lebesgue measure 
$\lambda$ by $h^{-1},\,$  i.e. $\mu(A)=\lambda(hA)$. 
This means that  $h$ is the distribution function for $\mu$, i.e. 
$$
    \mu([x_1,x_2))=h(x_2)-h(x_1)
$$ 
for any interval $[x_1,x_2) \subset S^1$, and hence
    \begin{equation} \label{distribution}
    \mu(B(x,r))=\,\Delta h(x,r)\, \overset{\text{def}}{=} \,h(x+r)-h(x-r). 
   \end{equation}
In particular, if $h$ is continuously differentiable, 
then for any $x$ and $r<1/2$,
$$
   2r\cdot \underset{[x-r,x+r]}{\min} |h'|  \;\le\; \mu(B(x,r)) \;\le\; 
   2r\cdot \underset{[x-r,x+r]}{\max} |h'|.
$$


\subsection{Proof of Proposition \ref{Diophantine}}\;
The following result was established by M.-R. Herman in \cite{H}:
any $C^{2+\varepsilon}$ circle diffeomorphism  whose  rotation number $\tau$ satisfies the Diophantine condition \eqref{dioph}
with some 
$K>0$ and $0<\delta <\varepsilon$  is conjugate to the rotation  
$R_\tau$ via a $C^1$ diffeomorphism. It follows that any $\Ci$
diffeomorphism with a Diophantine rotation number is smoothly 
conjugate to the corresponding rotation. This implies that there 
exist constants $m$ and $M$ such that 
$2mr\le \mu(B(x,r)) \le 2M r$ for all $x$ and all $r< 1/2$.
Therefore, $d_\mu(x)=1$ for every  $x\in S^1$, and hence
$\,\dim_H \mu =\ud_{B\,} \mu =\od_{B\,} \mu =1\,$
by Theorem \ref{Young}.   
$\QED$


\subsection{Proof of Theorem \ref{main}} 
Let $\tau$ be a Liouville number. First we note that to obtain the 
result for the case of $\beta=1$ it suffices to take 
$f=R_\tau$. The rotation $R_\tau$ preserves the Lebesgue measure,
which satisfies (1) and (2).
From now on we will assume that $0\le \beta <1$.
\vskip.1cm

We will obtain  the diffeomorphism $f$ as  a limit of a sequence 
of diffeomorphisms
$$
     f_n=h_n^{-1} \circ R_{\tau_n} \circ h_n,
$$
where  $h_n$ are $C^\infty$ diffeomorphisms of $S^1$ 
and $\tau_n$ are rational numbers that converges 
to $\tau$.  The sequences $\{h_n\}$ and $\{\tau_n\}$ will be defined inductively. Once $h_n$ is selected,
we will construct $h_{n+1}$ in the form
$$
  h_{n+1}=A_n \circ h_n,\;\; \text{ where }\;A_n=\text{Id}+a_n
$$
is a diffeomorphism and  $a_n$ is a $C^\infty$ periodic function.

The diffeomorphisms $f_n$ will converge in the 
$C^\infty$ topology,  and $h_n$ as well as $h_n^{-1}$
 will converge in $C^0$. The homeomorphism
 $h=\lim_{n\to \infty}h_n$ will give the distribution function 
 of the invariant measure $\mu$ for $f$.

\vskip.2cm
     
We will use the following  norm of diffeomorphisms 
 and the corresponding distance.

\begin{definition}\label{norm_def}
Let $g$ be a $C^n$  diffeomorphism of $\,[0,1]$. We define a norm of $g$
$$
   \|g\| _n^*= \max |g^{(i)}(x)|,
$$
where the maximum is taken over all $x$  and $\,0\le i \le n$,
and we denote
$$
      \|g\|_n=\max\, \{ \, \|g\|_n^*,\;  \|g^{-1}\|_n^* \,\}.
 $$
For two $C^n$  diffeomorphisms  $g_1$  and $g_2$, we set 
$$
 d_n(g_1, g_2)=\max \,\{ \,\|g_1-g_2\|_n^*,\;  \|g_1^{-1}-g_2^{-1}\|_n^*\,\}.
$$
\end{definition}
\vskip.2cm

In the three lemmas below
we will use  Fa\`a di Bruno's formula, which generalizes 
the chain rule to higher derivatives:
\begin{equation}\label{Bruno}
   \frac{d^n}{dx^n}\,f(g(x))= \sum c_{m_1,\dots, m_n}
   \,f^{(m_1+\dots +m_n)}(g(x)) 
   \prod_{j=1}^n \left(g^{(j)}(x)\right) ^{m_j},
\end{equation}
where the constants 
$c_{m_1,\dots, m_n}$ depend only on $m_1, \dots, m_n$,    
and the sum is taken over all $n$-tuples 
$(m_1, \dots ,m_n)$  of integers satisfying
\begin{equation}\label{m}
   1m_1+2m_2+ \dots +nm_n=n \quad\text{and}\quad 
   m_i\ge 0, \;\;i=1, \dots, n.
\end{equation}
\vskip.2cm

The following  lemma gives an estimate for the distance 
between two maps conjugate to two rotations via the 
same diffeomorphism.

\begin{lemma}\label{two rotations lemma}
Let $R_{\tau_1}$ and $R_{\tau_2}$ be two circle rotations, and let 
$h$ be a $C^{n+1}$ circle diffeomorphism. Then
\begin{equation}\label{two rotations}
  d_n( h^{-1} \circ R_{\tau_1} \circ h,\;  h^{-1} \circ R_{\tau_2} \circ h) 
  \;\le\; c_n  |\tau_1-\tau_2| \cdot  \|h\|_{n+1} ^{n+1},
\end{equation}
where the constant $c_n$ depends only on $n$.
\end{lemma}

\begin{proof} We will estimate 
  $\max | \frac{d^k}{dx^k}\, (h^{-1} \circ R_{\tau_1} \circ h)-
           \frac{d^k}{dx^k}\, (h^{-1} \circ R_{\tau_2} \circ h) |$
           for $0\le k \le n$.
For $k=0$ we have
$$
  \max |h^{-1}(h(x)+\tau_1)-h^{-1}(h(x)+\tau_2)| 
  \,\le \,\max |(h^{-1})'|\cdot |\tau_1-\tau_2| 
      \, \le\, \|h\|_1 \cdot |\tau_1-\tau_2|.
$$ 
 For $1\le k\le n\,$ formula \eqref{Bruno} yields
$$
  \begin{aligned}
  & \frac{d^k}{dx^k}\, (h^{-1} \circ R_\tau \circ h)(x)= 
    \frac{d^k}{dx^k} \left( h^{-1} (h(x)+\tau) \right)\\
   &\;=\;\sum c_{m_1, \dots ,m_k} 
   \,(h^{-1})^{(m_1+\dots +m_k)}(h(x)+\tau ) 
   \prod_{j=1}^k \left(h^{(j)}(x)\right) ^{m_j}.
\end{aligned}
$$
We estimate the difference between the corresponding terms
in $\frac{d^k}{dx^k}\, (h^{-1} \circ R_{\tau_1} \circ h)(x)$ and 
$\frac{d^k}{dx^k}\, (h^{-1} \circ R_{\tau_2} \circ h)(x).\;$
By \eqref{m}, $\;m_1+\dots + m_k \le k$, and we have
$$
\begin{aligned}
& \left| \,\prod_{j=1}^k \left(h^{(j)}(x)\right) ^{m_j} 
 \left( (h^{-1})^{(m_1+\dots +m_k)}(h(x)+\tau_1)
 -(h^{-1})^{(m_1+\dots +m_k)}(h(x)+\tau_2) \right) \right|  \\
 &\;\le\; \max \left| \,\prod_{j=1}^k \left(h^{(j)}(x)\right) ^{m_j} \right|
 \cdot \max |(h^{-1})^{(m_1+\dots +m_k+1)}|
 \cdot  |\tau_1-\tau_2| \\
  & \;\le\;  \|h\|_k^k  \cdot \|h\|_{k+1} \cdot |\tau_1-\tau_2| 
\;\le\;
 \|h\|_{k+1}^{k+1}\cdot |\tau_1-\tau_2| .
 \end{aligned}
  $$
It follows that for any $0\le k\le n$,
$$
  \max \left| \frac{d^k}{dx^k}\, (h^{-1} \circ R_{\tau_1} \circ h)(x)-
           \frac{d^k}{dx^k}\, (h^{-1} \circ R_{\tau_2} \circ h)(x)\right|
           \;\le \; c_n  |\tau_1-\tau_2| \cdot  \|h\|_{n+1} ^{n+1},
$$
where $c_n$ is the sum of the coefficients $c_{m_1\dots m_n}$
 in \eqref{Bruno}. 
Since  $\left( h^{-1} \circ R_{\tau} \circ h \right)^{-1}= 
 h^{-1} \circ R_{-\tau} \circ h,\;$ 
we have the same estimate  for the inverses of the functions,
and \eqref{two rotations} follows.
\end{proof}

 \vskip.2cm
 
When constructing the diffeomorphism $h_{n+1}=A_n \circ h_n$,
we will  choose the function $A_n$ in a specific form, 
as in the following lemma.

\begin{lemma}\label{a_s lemma}
Let $s$ and $\d$ be positive numbers such that $1/s$ is an integer and $\d<s/2$. Then there exists a $C^\infty$ diffeomorphism
$A=A_{s,\d}$ of $\,[0,1]$  such that 
\vskip.1cm
\begin{enumerate}
\item $A=\Id+a$, where $a=a_{s,\d}$ is a  non-negative $C^\infty$
function of period $s$,
\item  $A(0)=0$,  $\;A(\d)=s-\d$, and $\;A(s)=s$,
\vskip.1cm
\item $\d/(2s) \le A'(x)\le 2s/\d\;$ for all $x$,
\vskip.1cm
\item for each $n\ge 0$ there exists a constant $\rho_n$ 
that does not depend on $\d$ and $s$ such that 
$\; \|A\|_n \le \rho_n /\d^{n^2}$.
\end{enumerate}
\end{lemma}

Conditions (1), (2), and (3) guarantee that $A$ is a diffeomorphism. 

\begin{proof} First we construct the function $a$ 
on the intervals $[0,\d]$.
Let $g$ be a $C^\infty$ function on $[0,1]$ 
such that 
$g(x)=0$ in a neighborhood $[0,\e)$ of 0, 
$\,g(x)=1$ in a  neighborhood  $(1-\e,1]$ of 1,
and $0\le g'(x)\le 2$ for all $x$.
We obtain the function $a$ on $[0,\d]$ by rescaling $g$:
$$
 a(x)=(s-2\d)\, g(x/\d) \quad\text{for }x\in[0,\d].
$$
Then $a(x)=0$ in a neighborhood of $0$,  $a(x)=s-2\d$
in a neighborhood of $\d$, 
and $0\le a'(x)\le 2(s-2\d)/\d=2s/\d-4$.
This implies that $A(x)=x$ in a neighborhood of $0$,  
$\,A(\d)=s-\d$,  $\, A'(x)=1$ in a neighborhood of $\d$,
and $1\le A'(x)\le 2s/\d-3\le 2s/\d$
for all $x$ in $[0,\d]$.

Now we obtain the graph of the function $A$ on $[\d,s]$ by 
reflecting its graph on $[0,\d]$ with respect to the line $y=s-x,\;$ i.e.
\begin{equation}\label{inverse}
   A(x)=s-A^{-1}(s-x)\quad\text{for }x \text{ in } [s-\d,s].
\end{equation}  
 It follows from the symmetry of the graph that $A(x)=x$ 
 in a neighborhood of $s$ and 
 $1\ge A'(x)\ge \d/2s\,$ for all $x$ in $[s-\d,s]$.  
 On this interval, we set $a(x)=A(x)-x$. Clearly,
 $a(x)\ge 0$ for all $x$ and $a(x)=0$ in a neighborhood of $s$.
 Then we extend $a$ to $[0,1]$ by periodicity, and thus 
 obtain $A=\Id +a$ on $[0,1]$.
 Thus we have constructed a function $A$ satisfying (1), (2), and (3).
 \vskip.2cm
Now we will verify (4) for  $A$ on $[0,\d]$.  Then \eqref{inverse}
 will imply that (4) is also satisfied  for $A$ on $[\d,s].\;$ Since  
$\, \max_{[0,\d]} |a^{(n)}| \le  \max_{[0,1]} |g^{(n)}|\, /\delta^n,\,$
we have 
\begin{equation}\label{kappa}
\max_{[0,\d]} |A^{(n)}| \;\le\;  (\max_{[0,1]} |g^{(n)}|+1)\, /\delta^n
\;\overset{\text{def}}{=}\; \kappa_n/ \d^n\; \;
 \text{ for any } n\ge 0.
\end{equation}
 
Let $G$ be the inverse function for $A|_{[0,\d]}$.  We will show using induction 
that for any $n\ge 0$ there exists a constant $\xi_n$ independent of
 $\d$ and $s$ such that 
$$ 
     \max_{[0,1]} |G^{(n)}| \le \xi_n /\d^{n^2} \quad 
     \text{and }\; \xi_n\ge \xi_{n-1}\ge \dots \ge\xi_0. 
$$ 
Clearly, $G(x)\le 1$ and $\,G'(x) \le1 \le 1/\d$ for all $x$.
Thus the statement  holds for $k=0$ and $k=1$. Suppose that 
it holds for all $0\le k \le n-1$.
For $n\ge 2,$ $\;\frac{d^n}{dx^n}\,G(A(x))=\frac{d^n}{dx^n}\,x=0,\,$
and it follows from  \eqref{Bruno} that
$$
   G^{(n)}(A(x))\cdot (A'(x))^n= -\sum c_{m_1, \dots, m_n}
   \,G^{(m_1+\dots +m_n)}(A(x)) 
   \prod_{j=1}^n \left(A^{(j)}(x)\right) ^{m_j},
$$
where the sum is taken over all $n$-tuples $(m_1, \dots m_n)$ 
such that  $1m_1+\dots +nm_n=n$ and $m_1\ne n$.
This implies that $m_1+\dots +m_n \le n-1$,  and hence
 $$ 
    | G^{(m_1+\dots +m_n)}(A(x)) | 
    \;\le\;  \xi_{n-1} /\d^{(n-1)^2}
 $$
by the induction assumption. Also,
$$ 
     \left| \prod_{j=1}^n \left(A^{(j)}(x)\right) ^{m_j} \right| \le\,
      \prod_{j=1}^n \left(\kappa_j/\d^j\right) ^{m_j} =\,
     \left(\prod_{j=1}^n  \kappa_j ^{m_j}\right) / \d^{1m_1+2m_2+\dots +nm_n}\,\le\,
     \left( \prod_{j=1}^n  \kappa_j ^n\right) / \d^n.
 $$
 Using the estimates above, we obtain
$$
  |G^{(n)}(A(x))\cdot (A'(x))^n| \;\le\; \xi_n /\d^{(n-1)^2+n} \;\le\;
   \xi_n/ \d^{n^2},
$$
where $\xi_n \ge \xi_{n-1}$ is a constant independent of 
$s$ and $\d.\;$ Since $A'(x)\ge 1$ on $[0,\d],$ it follows that
$$
  |G^{(n)}(A(x))| \le 
  \xi_n/ \d^{n^2} \quad\text{for all } x\in [0,\d].
$$
Let $\,\rho_n=\max \,\{\kappa_0, \dots , \kappa_n, \,\xi_n \},\,$ 
where $\kappa_0, \dots \kappa_n$ are as in \eqref{kappa}.
Then  
\vskip.1cm
$\,\|A\|_n^*\le \rho_n/ \d^n \le \rho_n/ \d^{n^2},
\;\; \|A^{-1}\|_n^*=\|G\|_n^*\le \rho_n/ \d^{n^2},
\; \text{ and thus } \|A\|_n\le \rho_n/ \d^{n^2}.
$
\end{proof}
\vskip.3cm

\newpage

\begin{lemma}\label{norm_lemma} 
Let $A=A_{s,\d}$ be a function as in  Lemma \ref{a_s lemma}. 
Then for any $n\ge 0$ and any $C^n$ diffeomorphism $h$
\begin{equation}\label{norm}
  \|A \circ h \|_{n} \,\le\, \tilde{c}(h, n)/ \d^{n^2},
 \end{equation} 
where the constant $\tilde{c}(h,n)$ depends on $h$ and $n$, but not on $\d$ and $s$.
\end{lemma}

\begin{proof} Let $0\le k\le n$.
 Each term of the sum representing $(A\circ h)^{(k)}$
is a product of a derivative of $A$ of order at most $k$
and at most $k$ derivatives of $h$, see \eqref{Bruno}. 
Therefore, each term can be estimated by 
$\| A \|_k \cdot \|h\|_k^k \,\le \| A \|_n \cdot \|h\|_n^n,\;$ and hence 
$$
   \|A \circ h\|_n^*  \,\le\,  c_n \| A \|_n \cdot \|h\|_n^n \,\le\, 
   c_n \|h\|_n^n \cdot \rho_n/\d ^{n^2},
$$
where $c_n$ is the sum of the constants $c_{m_1, \dots, m_n}$ 
in \eqref{Bruno}
and $\rho_n$ is as in Lemma \ref{a_s lemma} (4).  Each term 
of $(h^{-1}\circ A^{-1})^{(k)}$ can be estimated as follows:
$$
\begin{aligned}
  & | (h^{-1})^{(m_1+\dots +m_k)}(A^{-1}(x)) 
   \prod_{j=1}^k \left( (A^{-1})^{(j)}(x)\right) ^{m_j} | \le
   \|h\|_k \cdot (\rho_1/\d^{1^2})^{m_1}  \dots (\rho_k/\d^{k^2})^{m_k}\\
   &\le \|h\|_k \cdot \rho_1\dots \rho_k /
   \d^{1^2m_1+2^2m_2+ \dots +k^2 m_k} \le 
   \|h\|_k \cdot \rho_1\dots \rho_k / \d^{k^2} 
   \le \|h\|_n \cdot \rho_1\dots \rho_n / \d^{n^2} 
\end{aligned}
$$
since $1^1m_1+2^2m_2+ \dots +k^2m_k
\le k(1m_1+2m_2+ \dots +km_k)=k^2.$
It follows that 
$$
   \| h^{-1}\circ A^{-1} \|^*_n \,\le\, c_n \|h\|_n \cdot  \rho_1\dots \rho_n / \d^{n^2}.
$$
Thus  $\|A \circ h \|_{n} \,\le\, \tilde{c}(h, n)/ \d^{n^2}$.
\end{proof}

\vskip.3cm

For the rest of the proof we will consider  the cases
  of $\beta=0$ and of $0<\beta<1$ separately.
  \vskip.2cm
  
  \begin{flushleft}
 {\bf The case of $\beta=0$}.
 \end{flushleft}

We will construct the sequences $\{\tau_n\}_{n=1}^{\infty}$ 
and $\{h_n\}_{n=1}^{\infty}$ 
inductively. Let $h_1$ be the identity map,  $\tau_1$ be 
a rational number close to $\tau$, and $\,s_1=1/2$. 
Suppose that a number 
$\tau_{n-1}=p_{n-1}/q_{n-1}$, a function $a_{n-1}$ 
of period $s_{n-1}$, and hence  diffeomorphisms
 $A_{n-1}=\Id+a_{n-1}$ and $h_n=A_{n-1}\circ h_{n-1}$ are selected.
  We denote
 $$
    M_n=\max_{[0,1]} h_n' \quad \text{and} \quad
    m_n=\min_{[0,1]} h_n'=1/{\max}_{[0,1]} (h_n^{-1})'.
 $$
 Clearly, $M_n\ge1$ and  $0< m_n\le1$ 
for all $n\ge 1$, and $M_n\to\infty$, $m_n\to 0$ as $n\to \infty$.
\vskip.1cm

 We choose a  rational number $\tau_n=p_n/q_n$, 
numbers  $s_n$  and $\delta_n$, 
a function $a_n$ of a period  $s_n$, and a function $A_n$ such that 

\begin{equation}\label{choice}
\begin{aligned}
\text{(i)} \;\;\; & |\,\tau-\tau_n\,| \le |\,\tau-\tau_{n-1}\,|,\\
\text{(ii)} \;\;\, &|\,\tau-\tau_n\,| = |\,\tau-p_n/q_n\,| \le 1/q_n^{3n^4}, \\
\text{(iii)}\;\, &  q_n \ge\max \,\{\, 1/s_{n-1},\;\, 1/m_n, \;\, (3M_n)^n,
\;\, c_n,\;\, \tilde{c}\,(h_n,n+1) \,\}, \\
    & \text{ where } c_n \text{ is as in \eqref{two rotations}$\;$ and }\; 
     \tilde{c}\,(h_n,n+1)\text{ is  as in \eqref{norm}}, \hskip1cm\\
\text{(iv)}\;\,  &s_n=s_{n-1}/q_n, \; \text{ hence } s_n \le s_{n-1}^2
        \;\text{ and }\; s_n \le 1/2^n, \\
\text{(v)}\;\;\,  &\d_n = s_n^n , \; \text{ in particular, } \d_n < s_n/2, \\
\text{(vi)}\;\,  & a_n\text{ and }A_n \text{ are as in Lemma \ref{a_s lemma} with } \d=\d_n \text{ and } s=s_n.
\end{aligned}
\end{equation}
Conditions (i), (ii), and (iii) can be satisfied since $\tau$ is a Liouville number. Condition (iv) ensures that  the maps
 $A_n$ and $R_{\tau_{n}}$   commute, and hence
 $$
 \begin{aligned}
  &h_n^{-1} \circ R_{\tau_n} \circ h_n \;= \;
  h_n^{-1} \circ A_n^{-1} \circ A_n \circ R_{\tau_n} \circ h_n \\ 
  & = \, h_n^{-1} \circ A_n^{-1} \circ R_{\tau_n}  \circ A_n\circ h_n \;=\;
  (A_n \circ h_n)^{-1} \circ R_{\tau_n} \circ (A_n \circ h_n) 
  \end{aligned}
 $$
 Conditions (iii), (iv), and (v) imply  that 
 $\,1/\d_n = 1/s_n^n =q_n^n/ s_{n-1}^n \le q_n^{2n}$.
Using this
 as well as \eqref{two rotations} and  \eqref{norm} we obtain 
 \begin{equation}\label{convergence}
\begin{aligned}
& d_n (f_{n+1}, f_n) \, = \,
    d_n \left( h_{n+1}^{-1} \circ R_{\tau_{n+1}} \circ h_{n+1},\;
    h_n^{-1} \circ R_{\tau_n} \circ h_n \right) \\
 & = d_n \left(  
 (A_n \circ h_n)^{-1} \circ R_{\tau_{n+1}} \circ (A_n \circ h_n),\;\;
 (A_n \circ h_n)^{-1} \circ R_{\tau_n} \circ (A_n \circ h_n) \right)\\
 &\le c_n  |\tau_{n+1}-\tau_n | \cdot \|A_n \circ h_n \|_{n+1}^{n+1}
 \;\le \;2 c_n  |\tau-\tau_n | \cdot  
 \left(\tilde{c}\,(h_n,n+1) /\d_n^{(n+1)^2}\right)^{n+1} 
\\
 &\le \; 2q_n (1/q_n^{3n^4})  \left( q_n \cdot q_n^{2n(n+1)^2} \right)^{n+1}
 \,\le\; 2 (1/q_n^{3n^4}) \, q_n^{2n(n+1)^3+n+2} \;\le \;1/2^n
\end{aligned}
\end{equation}
for all sufficiently large $n$.
 Since $d_n (f_{m+1}, f_m) \le d_m (f_{m+1}, f_m)\le 1/2^m$ 
 for $m\ge n$, it follows that the sequence $\{f_n\}$ converges 
 in the $C^n$-topology  for any $n$, i.e. it converges in the 
 $C^\infty$ topology.
  \vskip.2cm
  
  Now we will establish the convergence of the diffeomorphisms 
  $h_n$. We recall that $A_n=\text{Id}+a_n$ is a  diffeomorphism, 
 where $a_n$ is a $C^\infty$ function of period $s_n$, 
 and $1/s_n$ is an integer.
 It follows that $\max_{[0,1]} |A_n-\text{Id}\,| \le s_n$ and 
 $\max_{[0,1]} |A_n^{-1}-\text{Id}\,| \le s_n$.
 Since $h_{n+1} = A_n \circ h_n$, we estimate
 $$
    \max_{[0,1]} |\,h_{n+1} - h_n| 
            \,=\, \max_{[0,1]} |\,(A_n-\text{Id})\circ h_n | \,\le \, s_n, 
         \quad \text{and}
  $$
  $$          
    \max_{[0,1]} |h_{n+1}^{-1} - h_n^{-1}|  = 
    \max_{[0,1]} |h_{n}^{-1}\circ (A_n^{-1} -\text{Id}) |  \le 
    \max_{[0,1]} (h_n^{-1})' \cdot s_n = (1/m_n) s_n \le s_{n-1}
 $$  
by \eqref{choice} (iii) and (iv). This implies  that  $\,d_0(h_{n+1}, h_n) \le s_{n-1}\,$, and  since $s_n \le 1/2^n$, it follows that  the sequence of  
 diffeomorphisms $\{h_n\}$ converges to a homeomorphism $h$ 
with respect to the distance $d_0$. Moreover,  since 
 $\,s_{n} \le s_{n-1}^2$ we have
\begin{equation}\label{d_0(h, h_n)}
\max_{[0,1]} |h-h_n| \;\le\; {\sum}_{k=n}^\infty s_k \;\le \; 2s_n.
 \end{equation} 
  
  \vskip.2cm
 
 Now we will prove the dimensional properties of the invariant 
 measure $\mu$  with distribution function $h$.  
By the construction, most of the growth of $A_n$ on the 
interval $[0,1]$ is concentrated on the union of intervals 
$[\,is_n,\, is_n+\d_n]$. Let
\begin{equation}\label{E_n}
    \tilde{E}_{n}={\bigcup} _{i=0}^{(1/s_n)-1} [\,is_n,\, is_n+\d_n]
    \quad \text{and}\quad    E_{n}=h_{n}^{-1}(\tilde{E}_n).
\end{equation}
 Then   $h_{n+1}(E_{n})=(A_n\circ h_n)(h_n^{-1}(\tilde E_n))=
 A_n(\tilde{E}_n)$.
\vskip.1cm

Since $\tilde E_n$ consists of $1/s_n$ intervals $[\,is_n,\, is_n+\d_n]$
and $ A( is_n+\d_n)-A(is_n)=s_n-\d_n$, the total growth 
of $A_n$ on $\tilde{E}_{n}$ is 
 $$
    (s_n-\d_n)(1/s_n) =1-\d_n/s_n = 1-s_n^n/s_n
    \ge 1-s_n^{n-1} \ge 1-s_n,
 $$
and  the total growth of $h_{n+1}$ on $E_n$ is the same.
By \eqref{d_0(h, h_n)},  
$\;\Delta h \ge \Delta h_n - 4s_{n+1}$ on each of the $1/s_n$  
intervals in $E_n$.  Since $s_{n+1} \le s_n^2\,$ and 
$s_n\le 1/2^n$  by  \eqref{choice} (iv), we estimate that 
the total growth of $h$ on the set $E_{n}$ is at least 
$$
    1-s_n-4s_{n+1}(1/s_n) \;\ge\; 1-s_n-4s_n^2/s_n \;=\;
     1-5s_n \;\ge\; 1-5/2^n.
$$
Thus, for the measure $\mu$ with distribution function $h$,
\begin{equation}\label{mu(E_n)}
     \mu(E_{n}) \ge 1-5/2^n.
\end{equation}
\vskip.2cm

 Now we will show that $\underline{d}_\mu(x)$, the lower 
 pointwise dimension of $\mu$ at $x$, is 0   for $\mu$-almost every $x$.
We recall that $m_n=\min_{[0,1]} h_n'.\; $     
 The length of each interval $I$ in the set $E_{n}$ is bounded 
above by $\d_n/m_n$, since the length of  $h_n(I)$ is $\d_n$.
Let 
 $$
   r_n=\d_n /m_n=s_n^n/m_n .
  $$
Let $x$ be a point in $E_n$. Then the interval $[x-r_n, x+r_n]$
contains one of the intervals in $E$ and hence 
$\Delta h_{n+1}(x,r_n)\ge s_n-\d_n.\;$ 
It follows from \eqref{d_0(h, h_n)} that
$$  
 \Delta h(x,r_n) \,\ge\, \Delta h_{n+1}(x,r_n)-4s_{n+1} \,\ge\, 
 s_n-\d_n -4s_{n+1} \,\ge\, s_n/2
$$  
for all sufficiently large $n$ since $\,\d_n=s_n^n\,$ and 
$\,s_{n+1}\le s_n^2.\;$ Therefore
$$
 \frac{\log \Delta h(x,r_n)}{\log r_n} \;\le\;
  \frac{\log (s_n /2)}{\log r_n} \;= \;\frac{\log (s_n/ 2)}{\log (s_n^n/ m_n)}   \;\le\; \frac{2}{n}. 
 $$ 
The last inequality is equivalent to $\,s_n\le m_n^{2/n}/2\,$, and it follows  from \eqref{choice} that 
$s_n=s_{n-1}/q_n \le  s_{n-1}m_n \le m_n/2$.
\vskip.1cm

Thus for any sufficiently large $n$ there exists $r_n>0$ such that 
$$
 \frac{\log \mu(B(x,r_n))}{\log r_n} 
 \;=\; \frac{\log \Delta h(x,r_n)}{\log r_n}\;\le \; \frac{2}{n}
 \quad \text{for any } x\in E_n,
$$
 and $r_n\to 0$ as $n\to \infty$.
  Let $x$ be a point in $[0,1]$. If follows  that 
 $\underline{d}_{\mu}(x)=0$ provided that 
 for any $m$ there exist $n\ge m$ such that $x\in E_n$.
   Otherwise, $x$ is in $J=\bigcup_{m=1}^{\infty}
     \bigcap_{n=m}^{\infty}([0,1] - E_n)$.
It follows from   \eqref{mu(E_n)} that 
     $\mu(\bigcap_{n=m}^{\infty}([0,1] - E_n))=0$
     and hence $\mu(J)=0$. We conclude that 
\begin{equation} \label{lower 0}
   \underline{d}_{\mu}(x)=0 \quad\text{for }\mu\text{-almost every  }  
    x\in S^1.
\end{equation}

 \vskip.2cm
 Now we will show that $\overline{d}_\mu(x)$, the upper 
 pointwise dimension of $\mu$ at $x$, equals 1 for 
 $\mu$-almost  all $x$. We recall that $M_n=\max_{[0,1]} h_n',\; $
and hence $\Delta h_{n}(x,r)\le 2rM_n$.
We take $\tr_n=(3M_n)^{-n}$ and note that by \eqref{choice}
$s_n=s_{n-1}/q_n < 1/q_n\le 1/ (3M_n)^{n} \le \tr_n$.
It follows that 
$$
\Delta h(x,\tr_n) \le \Delta h_{n}(x,\tr_n)+4s_{n} 
\le 2\tr_n M_n + 4\tr_n \le 3M_n \tr_n
$$
for all sufficiently large $n$. Hence for all $x$,
$$
\frac{\log \mu(B(x,\tr_n))}{\log \tr_n} \,=\,
 \frac{\log \Delta h(x,\tr_n)}{\log \tr_n} \, \ge\,
 \frac{\log (3M_n\tr_n)}{\log \tr_n} \,=\,
  1+\frac{\log (3M_n)}{\log ((3M_n)^{-n})} = 1-\frac{1}{n}.
$$ 
Clearly, $\tr_n \to 0$ as $n\to \infty$, and  we conclude that  
$\overline{d}_\mu(x) \ge 1$
for all $x$. Since  $\mu$ is a Borel probability measure on $S^1$,
$\overline{d}_\mu(x) \le 1$ for $\mu$-almost every $x$ 
(\cite{KS} Lemma 2.1). Thus $\, \overline{d}_{\mu}(x)=1\,$
for  $\mu$-almost  all  $x$.  
 Combining this with \eqref{lower 0} we obtain
$$
   \underline{d}_{\mu}(x)=0 \quad \text{and}\quad  
   \overline{d}_{\mu}(x)=1 
   \quad\text{for }\mu\text{-almost every  }    x.
$$
This completes the proof of the first statement of Theorem 
\ref{main} for the case of $\beta=0$. 
\vskip.3cm


Now we will prove the results for the box and Hausdorff 
dimensions of $\mu$ (see Section \ref{Hausdorff, box} for the definitions).
First we show that $\od_{B\,} \mu=1.$
For $\tr_n$ as above, we have
$$
  \mu(B(x,\tr_n))\; \le \;
  \tr_n^{\,1-1/n} \;
  \text{ for any }x\in S^1. 
  $$
Let $Z$ be a set in $S^1$ with $\mu(Z) > 0.\,$  Then at least $\mu(Z)\cdot \tr _n^{\,-(1-1/n)}$ balls of radius 
  $\tr_n$ are needed to cover $Z$. Thus, 
  $$
    \frac{\log N(Z,\tr_n)}{\log (1/\tr_n)}   \; \ge \; 
     \frac{\log\, ( \mu(Z)\,\tr _n^{\,-(1-1/n)})}{-\log \tr_n} \;= \;
     1-\frac{1}{n}-\frac{\log \,\mu(Z)}{\log \tr_n} 
     \;\underset{n\to\infty}{\longrightarrow} 1.
 $$
 Since $\tr_n \to 0$ as $n\to\infty$, this implies that  that 
 $\od_B(Z)=1$. Thus, $\od_B(Z)=1$ for any set $Z$ 
 with $\mu(Z)>0$, and hence $\od_{B\,}\mu = 1$ by the
 definition \eqref{dim_mu}. 
\vskip.1cm 

Now we prove  that 
 $\ud_{B\,} \mu$,  the lower box dimension 
of $\mu$,  equals 0. 
Let $\,G_k=\bigcap_{n=k}^\infty E_n.\,$
By \eqref{mu(E_n)}, $\,\mu(E_{n}) \ge 1-5/2^n\,$
for each $n$, and hence 
$$
  \mu(G_k) \ge 1-5/2^{k-1}  \to 1\;\text{ as }\;k\to\infty.
 $$ 
We recall that for each $n$ the set $E_n$ consists 
of $1/s_n$ intervals of length at most $r_n$,
and $\log s_n /\log r_n\to 0$ as $n\to \infty$.
This implies that each $E_n$, and hence $G_k$,
can be covered by at most $1/s_n$ balls of diameter $r_n$, 
i.e. $N(G_k,r_n) \le 1/s_n$. Therefore,
$$
\ud_{B\,} G_k = \liminf_{\e\to 0}
\frac{\log N(G_k,\e)}{\log (1/\e)} \;\le \;
\underset{n\to \infty}{\lim \inf}\; \frac{\log N(G_k,r_n)}{\log (1/r_n)} \;\le\;
\underset{n\to \infty}{\lim}\; \frac{\log  s_n}{\log r_n} =0
$$
for any $k>0$.  Thus for any $\e>0$ there exists a set $G$
such that $\mu(G) > 1-\e$ and $\ud_{B\,} G=0$, which implies 
that $\ud_{B\,} \mu=0$ by the definition \eqref{dim_mu}. 
And since $0\le \dim_H \mu \leq \ud_{B\,} \mu$, 
it follows that the Hausdorff dimension of $\mu$ is also 0.
\vskip.2cm

 This completes the proof of the theorem for the case of $\beta=0$.
\vskip.3cm
 
 \newpage
 
 \begin{flushleft}
 {\bf The case of $\,0<\beta<1$}.
 \end{flushleft}
  
The proof for this case uses the same approach as the proof for
the case of $\beta=0$. However, some modifications are
needed to ensure that the lower pointwise dimension is $\beta>0$.
\vskip.1cm

Let $\,\gamma=(1-\beta)/2,\,$ then $\,\gamma >0\,$ and 
$\,\beta+\gamma/n < 1\,$ for any $n\ge1$. We will construct the sequences $\{\tau_n\}_{n=1}^\infty$ and $\{h_n\}_{n=1}^\infty$ 
inductively. Let $h_1$ be the identity map,  $\tau_1$ be 
a rational number close to $\tau$, and $\,s_1$
be a number such that $1/s_1$ is an integer  and 
$\d_1=s_1^{1/(\beta+\gamma)}< s_1/2$. 
Suppose that  $\tau_{n-1}=p_{n-1}/q_{n-1}$, a function 
$a_{n-1}$  of period $s_{n-1}$, and hence 
 $A_{n-1}$ and $h_n$ are selected.  As before, we set 
 $ \,M_n=\max_{[0,1]}h_n'$ and $m_n=\min_{[0,1]}h_n'.$ 
  We  choose  numbers
 $\tau_n=p_n/q_n$, $s_n$  and $\delta_n$, and
functions $a_n$ of a period  $s_n$ and $A_n$ such that 
\begin{equation}\label{choice_beta}
\begin{aligned}
\text{(i)} \;\;\; & |\,\tau-\tau_n\,| \le |\,\tau-\tau_{n-1}\,|,\\
\text{(ii)} \;\;\, &|\,\tau-\tau_n\,| = |\,\tau-p_n/q_n\,| \le 1/q_n^{n^4}, \\
\text{(iii)}\;\, &  q_n \ge\max \,\{\, 1/s_{n-1},\;\, 1/m_n, \;\, (3M_n)^n,\;\,
c_n,\;\, \tilde{c}\,(h_n,n+1) \,\}, \; \\
    &\text{ where } c_n \text{ is as in \eqref{two rotations} }
     \text{and } \tilde{c}\,(h_n,n+1)\text{ is  as in \eqref{norm}},\\
\text{(iv)}\;\,  &s_n=s_{n-1}/q_n, \; \text{ hence } s_n \le s_{n-1}^2
        \;\text{ and }\; s_n \le 1/2^n, \\
\text{(v)} \;\;\, & s_n\le 2^{-n-1}s_{n-1}, \; \text{ and }\;
  s_n^{\gamma}  \le \min \,\{\,1/(M_n+1)^n, \; (m_n/(2M_n))^n \,\}, 
  \hskip.5cm \\
\text{(vi)}\;\,  &   \d_n= m_n\, s_n^{1/(\beta+\gamma/n)} , \; \text{ hence } \d_n < s_n/2, \\
\text{(vii)}\; & a_n \text{ and }A_n \text{ are as in Lemma \ref{a_s lemma} with } \d=\d_n \text{ and } s=s_n.
\end{aligned}
\end{equation}
Condition (v) will be used only in the proof of Lemma \ref{Holder}
below. Clearly, it can be satisfied by choosing a sufficiently large $q_n$. We note that 
$q_n \ge \tilde{c}\,(h_n,n+1) \ge \| h\|_n \ge \max |h^{-1}|=1/m_n,\,$
and hence
  $$
  \frac{1}{\d_n}\; = \; \frac{1}{m_n s_{n}^{1/(\beta+\gamma/n)}}\;=\; 
  \frac{q_n^{1/(\beta+\gamma/n)}}
  {m_n s_{n-1}^{1/(\beta+\gamma/n)}} \;\le \;
   \frac{q_n^{2/(\beta+\gamma/n)}} {m_n } \;\le \;
  \frac{q_n^{2/\beta}}{m_n} \;\le\; q_n^{(2/\beta ) +1} 
  \;\le\; q_n^{3/\beta}.
  $$
Thus we can obtain an estimate similar to \eqref{convergence}:
$$
\begin{aligned}
& d_n (f_{n+1}, f_n) \, = \; \dots 
 \;\le \;2 c_n  |\tau-\tau_n | \cdot  
 \left(\tilde{c}\,(h_n,n+1) /\d_n^{(n+1)^2}\right)^{n+1} \; \le \; 
\\
 & 2q_n (1/q_n^{n^4})  \left( q_n (q_n^{3/\beta})^{(n+1)^2} \right)^{n+1}
 \;\le\; 2 (1/q_n^{n^4}) \, q_n^{(3/\beta)(n+1)^3+n+2} \;\le \;1/2^n
\end{aligned}
$$
for all sufficiently large $n$, and establish convergence of the 
sequence $\{f_n\}$ in the $C^\infty$ topology. The convergence 
of $\{h_n\}$ with respect to the distance $d_0$ can be shown as before.
\vskip.3cm

We define the sets  $\tilde{E}_n$ and $E_n$ be as in \eqref{E_n}. 
The total growth of $A_n$ on $\tilde{E}_n$, and hence the total 
growth of $h_{n+1}$ on $E_n$ is 
$$
   (s_n-\d_n)(1/s_n) \;=\; 
    1-\d_n/s_n  \;\ge\; 1- s_n^{(1/(\beta+\gamma /n))-1}.
$$
Since $E_n$ consists of $1/s_n$ intervals,
$\max|h- h_{n+1}| \le 2s_{n+1}$, and $ s_{n+1} \le s_n^2,\,$ 
the total growth of $h$ on  $E_{n}$ is at least 
$$
   1-s_n^{(1/(\beta+\gamma /n))-1} - 4s_{n+1}/s_n \;\ge\;
   1-s_n^{(1/(\beta+\gamma /n))-1} - 4s_n \;\ge\; 1-5s_n^\sigma
   \;\ge \;1-5/2^{n\sigma},
 $$ 
where $\sigma=\min \,\{\,(1/(\beta+\gamma)-1, \,1\,\} >0.$
Thus, for the measure $\mu$ with distribution function $h$,
\begin{equation}\label{mu(E_n)beta}
     \mu(E_{n}) \ge 1-5/2^{n\sigma}.
\end{equation}
\vskip.1cm

Now we will show that $\underline{d}_\mu(x)=\beta$ for almost all $x$.
We take 
\begin{equation} \label{r_n}
       r_n= \d_n/m_n =  s_n^{1/(\beta+\gamma/n)}, \quad\text{then}
\end{equation}
\begin{equation} \label{properties}
 r_n < s_n, \quad  s_n=r_n^{\beta+\gamma/n}, \quad\text{and hence}\quad 
       \frac{\log s_n}{\log r_n}  \,=\, \beta+ \frac{\gamma}{n}.
\end{equation}    
Let $x$ be a point  in $E_n$. Then
$$
   \Delta h(x, r_n) \ge \Delta h_{n+1}(x, r_n) - 4s_{n+1} \ge 
   s_n -\delta_n - 4s_{n+1}  \ge s_n/2
   $$
for all sufficiently large $n$. Therefore
 $$
  \frac{ \log \Delta h(x, r_n) }{\log r_n} \;\le\; 
   \frac{\log (s_n/2)}{\log r_n} \;=\; 
   \frac{\log s_n}{\log r_n} - \frac{\log2}{\log r_n} \;=\; 
   \beta - \frac{\log2}{\log r_n} \;\le\; \beta+\frac{1}{n}
 $$
since $r_n < s_n \le 1/2^n.\;$
It follows as before 
that $\underline{d}_\mu(x)\le \beta \,$ for $\mu$-almost every $x.$    
\vskip.2cm

To show that $\underline{d}_\mu(x)\ge \beta$ we will prove 
that the  function $h$
is H\"older continuous with exponent $\beta$. 
Then for any $x$ and $r>0$,
  $\;\Delta h(x,r) \le C(2r)^\beta$ and hence
$$  
  \frac{\log \mu(B(x,r))}{\log r}\;=\;\frac{\log \Delta h(x,r)}{\log r} 
  \;\ge\;  \frac{\log(C2^\beta)+\beta\log r}{\log r} 
  \;\underset{r\to 0}{\longrightarrow}\, \beta.
  $$
This implies that $\underline{d}_\mu(x)\ge \beta$ for all $x$,
and hence $\underline{d}_\mu(x)= \beta$ for $\mu$-almost every $x$.


\begin{lemma}\label{Holder}
The function $h$
is  H\"older continuous with exponent $\beta$.
\end{lemma}

\begin{proof} We  will  show using induction  that 
$$
\begin{aligned}
& |\,h_n(x)-h_n(y)\,|\,\le\, |x-y|^\beta \quad\text{for all } x,y 
   \text{ with }|x-y|\le s_{n-1}, \;\text{ and}\\
&|\,h_n(x)-h_n(y)\,|\,\le\, (3-2^{-n})|x-y|^\beta \quad\text{for all } x,y
 \text{ with }|x-y|\ge s_{n-1}.
\end{aligned}
$$
Clearly, this is true for $h_1=\Id$.  Suppose that it holds for $h_n$.
\vskip.1cm

Recall that $h_{n+1}=A_n \circ h_n$, and the diffeomorphism 
$A_n$ is of the form $\Id +a_n$, where $a_n\ge 0$ is a function
of period $s_n$. It follows that $x\le A_n(x)\le x+s_n$
for any point $x$ in $[0,1]$. Hence,  
$|A_n(x)-A_n(y)| \le |x-y| +s_n\;$ for any $x$ and $y$.
\vskip.2cm

If $|x-y| \ge s_{n-1}$ we obtain 
$$
\begin{aligned}
  &|h_{n+1}(x)-h_{n+1}(y)| \,=\, |A_n(h_n(x))-A_n(h_n(y))| \,\le\,
  |h_n(x)-h_n(y)|+s_n \\
  & \; \le (3-2^{-n})\,|x-y|^\beta +  2^{-n-1}\, |x-y|^\beta \,=\, (3-2^{-n-1})\,|x-y|^\beta
 \end{aligned} 
$$
since by \eqref{choice_beta} 
$\;s_n\le 2^{-n-1} s_{n-1} \le 2^{-n-1}|x-y| \le 2^{-n-1}|x-y|^\beta$ .
\vskip.2cm

If $ s_n \le |x-y| \le s_{n-1}$ we have
$$
|h_{n+1}(x)-h_{n+1}(y)| \;\le\; |h_n(x)-h_n(y)|+s_n \;\le\;
 |x-y|^\beta +s_n \;\le\; 2\,|x-y|^\beta .
$$

It follows from \eqref{choice_beta} and \eqref{properties} that  
$r_n^{\gamma/n}  \le s_n^{\gamma/n} \le 
\min \,\{1/(M_n+1), \; m_n/(2M_n)\}.$
We will use this together with the fact that 
 $\beta+\gamma/n <1$ for all $n$ in the two estimates below.
Suppose that $r_n\le |x-y|\le s_n$. Then 
$$
\begin{aligned}
 &|\,h_{n+1}(x)-h_{n+1}(y)\,| \;\le\; |h_n(x)-h_n(y)|+s_n  \;\le\;
 M_n |x-y| +s_n  \\
 & \;\le\;  (M_n+1)\,s_n  \;=\; (M_n+1)\,r_n^{\beta+\gamma/n} 
  \;\le\; r_n^\beta   \;\le\; |x-y|^\beta.
  \end{aligned}
$$

Finally, for $|x-y| \le r_n$ we have
$$
\begin{aligned}
& |\,h_{n+1}(x)-h_{n+1}(y)\,| \,\le\, \max h_{n+1}' \cdot |x-y| 
 \,=\,\max A_n' \cdot \max h_n' \cdot |x-y|  \\
 &\;\le\, \frac{2s_n}{\d_n}\cdot M_n \cdot |x-y| \;=\; 
 \frac{r_n^{\beta+\gamma/n}}{ r_n m_n} \cdot 2M_n \cdot |x-y| \;\le\; 
 \frac{2M_n}{m_n} \cdot |x-y|^{\beta+\gamma/n}  \\
 & \;\le\, \frac{2M_n}{m_n} \cdot r_n^{\gamma/n}\cdot |x-y|^{\beta} 
   \;\le\;  |x-y|^\beta.
    \end{aligned}
$$

Thus, each function $h_n$ satisfies  
$ |h_n(x)-h_n(y)|\,\le\, 3\,|x-y|^\beta$ for all $x$ and $y.$
Since the sequence $\{h_n\}$ converges to $h$, it follows that 
$$
  |h(x)-h(y)|\,\le\, 3\,|x-y|^\beta \quad \text{for all }x \text{ and }y.
 $$
\end{proof}


The proof of the fact that $\overline{d}_\mu(x)=1$ for $\mu$-almost every $x$
does not require any modifications. Thus
   $$
   \underline{d}_{\mu}(x)=\beta \quad \text{and}\quad  
   \overline{d}_{\mu}(x)=1 
   \quad\text{for }\mu\text{-almost every  }    x.
$$
 The same argument as for $\beta=0$ shows that $\od_{B\,} \mu =1$
for all $x$.
\vskip.1cm

 Now we show that $\dim_H\mu =\ud_{B\,} \mu =\beta$.
 Since $\underline{d}_{\mu}(x)=\beta$ for almost all $x$, 
 it follows from  Theorem \ref{Young} that $\dim_H\mu \ge \beta$.
 So it remains to show that $\ud_B \mu \le \beta$.
 As before, let $G_k=\bigcap_{n=k}^\infty E_n$. 
Since $G_k\subset E_n$ and $\;E_n$ consists of $1/s_n$ intervals of length at most $r_n$, we have $N(G_k,r_n) \le 1/s_n$.
As $\log s_n /\log r_n\to \beta$ as $n\to \infty$, we obtain
$$
\ud_{B\,} G_k \;\le \;
\underset{n\to \infty}{\lim}\; \frac{\log N(G_k,r_n)}{\log (1/r_n)} \;\le\;
\underset{n\to \infty}{\lim}\; \frac{\log  s_n}{\log r_n} =\beta
$$
for any $k>0$.  It follows from \eqref{mu(E_n)beta} that 
$\mu(G_k) \to 1$ as $k\to\infty$, and hence 
$\ud_{B\,} \mu \le \beta$. 

\vskip.1cm

 This completes the proof of the theorem.

\vskip1cm


\end{document}